\documentclass[11pt]{article}

\usepackage[margin=1in]{geometry}
\usepackage{amsmath}
\usepackage{amsfonts}
\usepackage{amsthm}
\usepackage{amssymb}
\usepackage{graphicx}
\usepackage{mathrsfs}
\usepackage{nicefrac}
\usepackage[T1]{fontenc}
\usepackage[ansinew]{inputenc}
\usepackage{lmodern}
\usepackage{enumitem}
\usepackage{natbib}
\bibpunct{(}{)}{;}{a}{,}{,}
\usepackage[colorlinks=true,linkcolor=blue,citecolor=blue,urlcolor=blue]{hyperref}
\usepackage[capitalise]{cleveref}

\newcommand{\ud}[1]{\, \mathrm{d}#1}
\newcommand{\R}{\mathbb{R}}
\renewcommand{\P}{ \mathbb P }
\newcommand{\EXP}{\mathbb{E}}

\DeclareMathOperator{\diag}{diag}
\DeclareMathOperator{\Diag}{Diag}
\DeclareMathOperator{\sign}{sign}
\DeclareMathOperator{\Conv}{Conv}
\newcommand{\pp}{\mathfrak{p}}
\newcommand{\ii}{\mathfrak{i}}

\theoremstyle{plain}
\newtheorem{theorem}{Theorem}
\newtheorem{lemma}{Lemma}
\newtheorem{proposition}{Proposition}
\newtheorem{corollary}{Corollary}

\newcommand{\introthmname}{}
\newtheorem*{introthm}{\introthmname}
\theoremstyle{definition}
\newtheorem{definition}{Definition}
\newtheorem{assumption}{Assumption}
\newtheorem{example}{Example}
\theoremstyle{remark}
\newtheorem{remark}{Remark}
\crefname{assumption}{Assumption}{Assumptions}
\crefname{example}{Example}{Examples}

\begin{document}

\title{Coexistence and extinction in stochastic Lotka--Volterra systems\\ with Volterra--Lyapunov stable interactions}
\author{Jorge Ramirez Osorio \thanks{Universidad Nacional de Colombia, Sede Medellín. \texttt{jmramirezo@unal.edu.co}}}
\date{\today}
\maketitle

\begin{abstract}
We consider stochastic Lotka--Volterra systems $\ud X=\Diag(X)(\alpha+AX)\ud t+\Diag(X)\,\Gamma\ud B$ with nondegenerate noise and a Volterra--Lyapunov stable interaction matrix $A$, that is, $DA+A^\top D\prec0$ for some positive diagonal matrix $D$. This class covers every self-limited predator--prey food web that is conservative in the sense of Volterra, in particular every such web whose interaction graph is a forest. Let $\beta=\alpha-\frac12\diag(\Gamma\Gamma^\top)$ be the vector of noise-corrected growth rates and $m=-A^{-1}\beta$. We show that the process has a stationary distribution under which all species have positive densities if and only if $m>0$. In that case the process is strongly stochastically persistent, and its stationary distribution is unique and has mean $m$. When $m>0$ fails, the species that survive are those present in the unique saturated nonnegative equilibrium $\bar x$ of the deterministic system $\dot x=\Diag(x)(\beta+Ax)$. Almost surely, the occupation measures converge to the stationary distribution of this subcommunity, and $\log X_i(t)/t$ converges to $(\beta+A\bar x)_i\le0$ for every species $i$. The proofs rely on the theory of invasion rates of Hening and Nguyen. For Lotka--Volterra systems, invasion rates are affine in the mean of the underlying measure. Combined with two quadratic identities, this provides explicit weights $Dm$ for their persistence theorem and identifies the limiting regime when persistence fails, without enumerating the invariant measures on the boundary.
\end{abstract}



\section{Introduction}\label{sec:intro}

A basic question in theoretical ecology is whether a community of interacting species can persist indefinitely in a randomly fluctuating environment. For populations subject to environmental white noise, a standard model is the stochastic Lotka--Volterra system
\begin{equation}\label{eq:SDE}
	\ud X_i(t) = X_i(t)\left(\alpha_i + \sum_{j=1}^n A_{ij} X_j(t)\right) \ud t + X_i(t) \sum_{j=1}^n \Gamma_{ij} \ud B_j(t),
\end{equation}
where $X_i(t)$ is the density of species $i$, $\alpha\in\R^n$ collects the intrinsic growth rates, $A\in\R^{n\times n}$ is the interaction matrix, $\Gamma\in\R^{n\times n}$ is the diffusion matrix, and $B$ is a standard $n$-dimensional Brownian motion. Modelling environmental fluctuations by multiplicative white noise goes back to \citet{may1973stability} and \citet{turelli1977random}, and the persistence of interacting populations in fluctuating environments has been an active subject since then; see \citet{schreiber2017coexistence} for a review of the current theory, and \citet{lande2003stochastic} for a textbook account of environmental stochasticity in population dynamics. In this setting, we say that the $n$ species exhibit \emph{invariant coexistence} if $X$ has an invariant probability measure $\pi$ with $\pi(\R^{n,\circ}_+)=1$, where $\R^{n,\circ}_+=(0,\infty)^n$, that is, if there is a stationary regime in which all species are present. The stronger notion of \emph{strong stochastic persistence} \citep{hening2018coexistence} requires, in addition, that $\pi$ be unique and that the law of $X(t)$ converge to $\pi$ from every interior initial condition.

In this paper we describe the long-run behavior of \cref{eq:SDE} from every interior initial condition, for a broad class of interaction matrices. For this class, we show that invariant coexistence and strong stochastic persistence are equivalent, and both are characterized by an explicit condition on the coefficients; when this condition fails, we identify the species that survive, the stationary regime in which they settle, and the exponential rates at which the other species go extinct. Our main assumption is that the interaction matrix $A$ in \cref{eq:SDE} is \emph{Volterra--Lyapunov stable}: there is a positive diagonal matrix $D=\Diag(d)$ such that $DA+A^\top D\prec0$. This class covers every self-limited predator--prey food web that is conservative in the sense of Volterra, in particular every such web whose interaction graph is a forest, as well as weak competition and weak mutualism; see \cref{sec:disc-scope}.
 
Our main tool is the mathematical theory of \emph{invasion rates} for stochastic population models, which has been developed mostly in the last fifteen years \citep{schreiber2011persistence,hening2018coexistence,hening2021general,hofbauer2022permanence}; our analysis relies mainly on the framework of Hening, Nguyen and coauthors \citep{hening2018coexistence,hening2021general,hening2022classification}. The invasion rates are the long-run per-capita growth rates $\lambda_i(\mu)$ of each species $i$ under an ergodic invariant measure $\mu$ supported on the boundary $\partial\R^n_+$ of the orthant, that is, under a stationary regime in which some species are absent. The main result of \citet{hening2018coexistence}, written in the equivalent weighted form of \citet[Lemma~4]{schreiber2011persistence}, states that the solution $X$ of \cref{eq:SDE} is strongly stochastically persistent if there are weights $p_i>0$ such that $\sum_ip_i\lambda_i(\mu)$ is bounded below by a positive constant over all such $\mu$. Their extinction theorem is also conditional: it applies once the boundary regimes that attract nearby solutions have been identified and all mixtures of the remaining ones have been shown to be repelled. Applying these results requires, in principle, computing the invasion rates of all ergodic invariant measures on all faces of the orthant. For Lotka--Volterra systems such as \cref{eq:SDE} this task becomes simpler because the invasion rates depend on a boundary measure only through its mean and these means solve linear systems; this was used for three species by \citet[Prop.~8.2]{hening2022classification} and for food chains by \citet{hening2018stochastic,hening2018persistence}. Still, the number of faces grows like $2^n$, and complete classifications of stochastic Lotka--Volterra systems are available only in low dimension or for special topologies \citep{rudnicki2003long,rudnicki2007influence, du2016conditions,hening2018stochastic,hening2018persistence,hening2022classification}; for three species, \citet{hening2022classification} classify all generic cases and conjecture that this extends to higher dimensions. We extend these classifications to Volterra--Lyapunov stable interaction matrices, with any number of species and any interaction network, by verifying the hypotheses of both the persistence and the extinction theorems of \citet{hening2018coexistence} directly from the coefficients, without enumerating the boundary regimes.

Related results have been obtained for other noise models: \citet{benaim2022stochastic} treat food chains with degenerate noise, and \citet{bourquin2021persistence} showed that for food chains with randomly switched coefficients persistence is equivalent to feasibility of the equilibrium of the averaged vector field. Closest to our setting, \citet{videla2022strong} extends the theorem of \citet{hening2018coexistence} to Lotka--Volterra systems driven by Brownian and L\'evy noise, for \emph{layered} communities: acyclic food webs with trophic levels and self-limitation. The analysis there uses weights $c_i>0$ with $c_i|A_{ij}|\ge c_j|A_{ji}|$ whenever species $j$ preys on species $i$, an inequality version of Volterra's conservativeness, and the weighted invasion condition is verified through explicit sufficient inequalities for webs with two and three species and for food chains, which are not sharp in the presence of jumps \citep[Sect.~3]{videla2022strong}. To our knowledge, no criterion that is both necessary and sufficient and stated directly in terms of the coefficients is available for general interaction networks, even with Brownian noise only.

Our approach is to use the theory of invasion rates to express the long-run properties of \cref{eq:SDE} in terms of the deterministic Lotka--Volterra system
\begin{equation}\label{eq:JansenODE}
	\dot x=\Diag(x)(\beta+Ax),
\end{equation}
obtained through the \emph{It\^o shift} $\alpha\mapsto\beta:=\alpha-\frac12\diag(\Sigma)$, which replaces the intrinsic growth rates by the noise-corrected growth rates. Here $\Sigma:=\Gamma\Gamma^\top$ is the covariance matrix of the noise, $\diag(\Sigma)$ is the vector of its diagonal entries, and the correction $-\frac12\Sigma_{ii}$ is the one that appears in It\^o's formula for $\log X_i$. Under Volterra--Lyapunov stability, the long-run behavior of \cref{eq:JansenODE} is completely understood: for every $\beta$ it has a unique saturated nonnegative equilibrium $\bar x$, at which no absent species can grow, and $\bar x$ attracts every solution with positive initial densities \citep{takeuchi1980existence}; in particular, \cref{eq:JansenODE} is permanent if and only if its interior equilibrium $m=-A^{-1}\beta$ is feasible, that is, $m>0$, in which case $m=\bar x$. We show that the vectors $m$ and $\bar x$ also characterize the long-run behavior of \cref{eq:SDE}: feasibility of $m$ is equivalent to invariant coexistence and, when it fails, $\bar x$ identifies the species that survive and the rates at which the others go extinct. More precisely, we prove the following two results, stated in full as \cref{thm:main,thm:saturated} in \cref{sec:main,sec:saturated}.

\renewcommand{\introthmname}{Characterization of invariant coexistence (\cref{thm:main})}
\begin{introthm}
	Assume that $DA+A^\top D\prec0$ for some positive diagonal matrix $D$ and that $\Sigma\succ0$. Then \cref{eq:SDE} has an invariant probability measure $\pi$ with $\pi(\R^{n,\circ}_+)=1$ if and only if $m:=-A^{-1}\beta>0$. In that case $X$ is strongly stochastically persistent: $\pi$ is the unique such measure, it is ergodic, its mean is $\int x\,\pi(\ud x)=m$, and the transition probabilities of $X$ converge to $\pi$ in total variation, exponentially fast, from every $x\in\R^{n,\circ}_+$.
\end{introthm}

\renewcommand{\introthmname}{Long-run behavior without invariant coexistence (\cref{thm:saturated})}
\begin{introthm}
	Under the same assumptions, suppose that $m>0$ fails. Then there exist a set of species $S\subsetneq\{1,\dots,n\}$ and an invariant probability measure $\pi^S$ of \cref{eq:SDE} supported on $\R^{S,\circ}_+$ (see \cref{eq:faces}) such that, for every $X(0)\in\R^{n,\circ}_+$, almost surely, the occupation measures $\frac1t\int_0^t\delta_{X(s)}\ud s$ converge weakly to $\pi^S$. Moreover, $S$ is precisely the set of species present in the unique saturated nonnegative equilibrium $\bar x$ of \cref{eq:JansenODE} and, almost surely, $\lim_{t\to\infty}\log X_i(t)/t=(\beta+A\bar x)_i\le0$ for every $i$. In particular, every species $i \notin S$ with $(\beta+A\bar x)_i<0$ goes extinct exponentially fast.
\end{introthm}

Both results rest on the fact that, for \cref{eq:SDE}, the vector of long-run growth rates under a stationary distribution $\mu$ is affine in its mean, $\lambda(\mu)=\beta+A\mu(X)=A(\mu(X)-m)$; see \cref{eq:def_lambda}. A species that is present under $\mu$ has zero long-run growth rate (\cref{lem:loglimit}). Therefore, for a stationary distribution under which all species are present, these zero growth rates are linear equations that force its mean to be $m$, even though the distribution itself is not explicit; since this mean is positive, $m>0$ is necessary for invariant coexistence. For sufficiency, we use the negative definiteness of $A$ to show that, when $m>0$, the vector $m$ provides the weights in the persistence condition of \citet{hening2018coexistence}: every stationary regime $\mu$ on the boundary satisfies $m^\top\lambda(\mu)>0$ (\cref{lem:weighted}). Thus, $m>0$ is equivalent to invariant coexistence. \Cref{thm:main} can then be applied to every subcommunity $J \subseteq \{1,\dots,n\}$, and all stationary regimes on the boundary can be characterized a posteriori (\cref{cor:sub}). For \cref{thm:saturated}, we use the extinction theory of \citet[Section~5]{hening2018coexistence} to show that every limit point $\mu$ of the occupation measures is invariant and satisfies $\lambda(\mu)\le0$. A strict concavity argument for the function $x^\top(\beta+Ax)$ shows that such a $\mu$ cannot spread over several faces and must be the stationary regime associated with $\bar x$ (\cref{lem:noninvadable}).

\Cref{thm:saturated} complements the extinction theorem of \citet[Theorem~1.3]{hening2018coexistence}. For the whole Volterra--Lyapunov stable class, \cref{lem:noninvadable} shows that there is a single attracting regime, given explicitly and the same for every interior initial condition. When no absent species has a zero invasion rate in that regime, the lemma also verifies the hypotheses of that theorem (\cref{rem:HN-extinction}); \cref{thm:saturated} covers, in addition, the borderline case, where that theorem does not apply. To our knowledge, no comparable description of the long-run behavior of stochastic Lotka--Volterra systems was available in arbitrary dimension. Taken together, \cref{thm:main,thm:saturated} show that, from every interior initial condition, the long-run time average of $X$ is $\bar x$ and the vector of long-run growth rates of $\log X_1,\dots,\log X_n$ is $\beta+A\bar x$ (\cref{sec:disc-sub}). In particular, environmental noise affects these long-run quantities only through the It\^o shift, while the stationary distributions themselves depend in general on the whole covariance matrix $\Sigma$.

The organization of this paper is as follows. \Cref{sec:model} states the assumptions and preliminaries, and \cref{sec:rates} computes the means of invariant measures on faces of the orthant. \Cref{sec:main} proves the characterization of invariant coexistence (\cref{thm:main}) and describes the stationary regimes of subcommunities. \Cref{sec:saturated} describes the long-run behavior when $m>0$ fails and proves \cref{thm:saturated}. \Cref{sec:discussion} discusses the scope of the assumptions, the Jensen effect of environmental noise, with an example in which noise on one species increases the number of surviving species, the relation with Hofbauer's criterion and the Lyapunov function of Volterra and Goh, and the expression of the long-run behavior of \cref{eq:SDE} in terms of the equilibria of \cref{eq:JansenODE}.

\section{Model, assumptions and preliminaries}\label{sec:model}

\subsection{Notation and assumptions}
 We denote $[n]=\{1,\dots,n\}$, the orthant by $\R^n_+=[0,\infty)^n$ and its interior by $\R^{n,\circ}_+=(0,\infty)^n$. Inequalities between vectors are componentwise; in particular $x>0$ means $x_i>0$ for every $i$, and $\mathbf 1=(1,\dots,1)^\top$. For a vector $x\in\R^n$, $\Diag(x)$ denotes the $n\times n$ diagonal matrix with diagonal entries $x_1,\dots,x_n$; for a matrix $M\in\R^{n\times n}$, $\diag(M)=(M_{11},\dots,M_{nn})^\top$ denotes the vector of its diagonal entries. For a symmetric matrix $M$, $M\prec0$ ($M\succ0$) means that $M$ is negative (positive) definite. We write $\|x\|_1=\sum_i|x_i|$ and, for a matrix $M$, $\|M\|_1=\max_j\sum_i|M_{ij}|$ for the induced norm. For $x\in\R^n_+$, $\P_x$ and $\EXP_x$ denote probability and expectation for the solution of \cref{eq:SDE} with $X(0)=x$, and $\P_\mu=\int\P_x\,\mu(\ud x)$ for a probability measure $\mu$ on $\R^n_+$. A probability measure $\mu$ on $\R^n_+$ is \emph{invariant} for $X$ if $X(0)\sim\mu$ implies $X(t)\sim\mu$ for every $t\ge0$, and an invariant probability measure is \emph{ergodic} if it gives measure $0$ or $1$ to every invariant Borel set \citep[Ch.~10]{kallenberg2002foundations}.

For $J\subseteq[n]$ let $\pp_{J}:\R^n\to\R^{|J|}$ be the projection onto the coordinates in $J$ and $\ii_{J}:\R^{|J|}\to\R^n$ the zero-padding inclusion, so that $\pp_{J}\circ\ii_{J}$ is the identity of $\R^{|J|}$. For a vector $x$ we write $x^J=\pp_{J}x$, and for a matrix $M$ we write $M^J=\pp_{J}\,M\,\ii_{J}$ for its principal submatrix with rows and columns in $J$. We denote the $J$-face and its relative interior by
\begin{equation}\label{eq:faces}
	\R^{J}_+:=\{x\in\R^n_+:\ x_i=0\ \forall i\notin J\}, \qquad
	\R^{J,\circ}_+:=\{x\in\R^{J}_+:\ x_i>0\ \forall i\in J\} .
\end{equation}
In particular $\R^{\emptyset}_+=\R^{\emptyset,\circ}_+=\{0\}$ and $\R^{[n],\circ}_+=\R^{n,\circ}_+$. The sets $\R^{J,\circ}_+$, $J\subseteq[n]$, partition $\R^n_+$, and the boundary is $\partial\R^n_+=\bigcup_{J\subsetneq[n]}\R^{J,\circ}_+$. We say that a probability measure $\mu$ on $\R^n_+$ is \emph{supported on} a Borel set $U\subseteq\R^n_+$ if $\mu(U)=1$.

The only assumption on the parameters of our model is:

\begin{assumption}\label{ass:VL}
	\begin{enumerate}[label=\textup{(\alph*)},ref=\theassumption-\alph*]
		\item\label[assumption]{ass:VL-a} The covariance matrix is nondegenerate: $\Sigma:= \Gamma\Gamma^\top \succ0$.
		\item\label[assumption]{ass:VL-b} The interaction matrix is \emph{Volterra--Lyapunov stable}:
		\begin{equation}\label{eq:VL}
			\text{there is} \quad d\in\R^{n,\circ}_+ \quad \text{such that} \quad DA+A^\top D\prec0, \quad\text{with}\quad D=\Diag(d).
		\end{equation}
	\end{enumerate}
\end{assumption}

We discuss the scope of \cref{ass:VL} in \cref{sec:disc-scope}.

\subsection{Consequences of \texorpdfstring{\cref{ass:VL}}{Assumption 1}}
First, we note that the matrix $D$ in \cref{eq:VL} can be taken to be the identity by a change of variables. Since $D$ is constant and diagonal, $\tilde X(t):=DX(t)$ solves \cref{eq:SDE} with the same $\alpha$ and $\Gamma$ and with community matrix $\tilde A:=AD^{-1}$, which satisfies
\begin{equation}\label{eq:VL-reduced}
	\tilde A+\tilde A^\top=D^{-1}\left(DA+A^\top D\right)D^{-1}\prec0 ,
\end{equation}
because the right-hand side is congruent to $DA+A^\top D$. This allows us \emph{from now on to take $D=I$ in \cref{ass:VL-b}, that is, to assume}
\begin{equation}\label{eq:assA}
	A+A^\top\prec 0,
\end{equation}
except for our main \cref{thm:main}, which is stated for a general $D$ and reduced to \cref{eq:assA} in its proof.

The following consequences of \cref{eq:assA} are used throughout.

\begin{remark}\label{rem:principal}
	The condition in \cref{eq:assA} implies that $A$ is negative definite, in the sense that
	\begin{equation}\label{eq:negdef}
		x^\top Ax=\frac12x^\top(A+A^\top)x<0, \quad \text{for all} \quad x \neq 0.
	\end{equation}
	The same holds for every principal submatrix, since $y^\top A^Jy=(\ii_Jy)^\top A\,(\ii_Jy)$ for $y\in\R^{|J|}$. Hence, for every nonempty $J\subseteq[n]$, $A^J$ is invertible, its diagonal entries are negative, and its eigenvalues have negative real parts, so $\sign(\det A^J)=(-1)^{|J|}$ and $-A$ is a P-matrix. Invertibility, the signs of the diagonal entries and the sign of $\det A^J$ are preserved by the scaling $A\mapsto AD^{-1}$, because $A^J=\tilde A^J\Diag(d^J)$ with $d^J=\pp_Jd>0$; so they also hold under \cref{ass:VL-b} with a general $D$.
\end{remark}

\subsection{Well-posedness and moments}
We use the framework of \citet{hening2018coexistence}, who study Kolmogorov systems $\ud X_i=X_if_i(X)\ud t+X_ig_i(X)\ud E_i$, where $E$ is a Brownian motion with covariance matrix $(\sigma_{ij})$. Our system \cref{eq:SDE} is of this form, with
\begin{equation}\label{eq:HNform}
	f(x)=\alpha+Ax,\qquad g_i\equiv1,\qquad E=\Gamma B,\qquad
	(\sigma_{ij})=\Sigma .
\end{equation}
Their standing Assumption~1.1 is verified in \citet[Example~1.1]{hening2018coexistence} for competitive Lotka--Volterra systems and for two-species predator--prey systems, and in \citet[Prop.~8.1]{hening2022classification} for three-species systems. The following lemma uses the same arguments to extend these results to our model \cref{eq:SDE} in any dimension.

\begin{lemma}\label{lem:HNass}
	\begin{enumerate}[label=(\arabic*),ref=\thelemma-\arabic*]
		\item\label[lemma]{lem:HNass-wp} \Cref{eq:SDE} satisfies Assumption~1.1 of \citet{hening2018coexistence} with $c=\mathbf 1$. Consequently for every $x\in\R^n_+$ there is a pathwise unique global strong solution with $X(0)=x$, $X$ is a Feller process on $\R^n_+$, and each $\R^{J,\circ}_+$, $J\subseteq[n]$, is invariant.
		\item\label[lemma]{lem:HNass-mom} Every invariant probability measure $\mu$ of $X$ has a finite first moment, $\int\|x\|_1\,\mu(\ud x)<\infty$.
	\end{enumerate}
\end{lemma}

\begin{proof}
	(1) Parts~(1) and~(2) of Assumption~1.1 of \citet{hening2018coexistence} hold because, by \cref{eq:HNform}, the matrix $(g_ig_j\sigma_{ij})$ equals $\Sigma\succ0$ and $f$ is affine. Part~(3) requires $\gamma_b>0$ with
	\[
		\limsup_{\|x\|_1\to\infty,\ x\in\R^n_+}
		\left[\frac{\mathbf 1^\top\Diag(x)(\alpha+Ax)}{1+\mathbf 1^\top x}
		-\frac12\frac{x^\top\Sigma x}{(1+\mathbf 1^\top x)^2}
		+\gamma_b\left(1+n+\|\alpha+Ax\|_1\right)\right]<0 .
	\]
	The second term is nonpositive. For $x\in\R^n_+$, $\mathbf 1^\top x=\|x\|_1$, and, since $A$ is negative definite by \cref{eq:negdef}, $x^\top Ax\le-\kappa\|x\|_2^2\le-\frac{\kappa}{n}\|x\|_1^2$, with $\kappa:=-\frac12\lambda_{\max}(A+A^\top)>0$, where $\lambda_{\max}$ denotes the largest eigenvalue. Hence the first term is at most $-\frac{\kappa}{n}\|x\|_1+O(1)$, while the third is at most $\gamma_b\|A\|_1\|x\|_1+O(1)$. With $\gamma_b:=\kappa/(2n\|A\|_1)$ the bracket is at most $-\frac{\kappa}{2n}\|x\|_1+O(1)$, so the $\limsup$ is $-\infty$. Well-posedness, the Feller property and the invariance of $\R^{J,\circ}_+$ for every $J$ follow from \citet[Lemmas~3.1 and~3.2]{hening2018coexistence}.

	(2) By \citet[Lemma~3.3]{hening2018coexistence}, there is $H_2>0$ such that every invariant probability measure $\mu$ satisfies $\int\left(1+\sum_i|f_i(x)|\right)\mu(\ud x)\le H_2$ (their bound carries an additional factor $(1+\mathbf 1^\top x)^{\delta_0}\ge1$). Since $A$ is invertible (\cref{rem:principal}), writing $x=A^{-1}\left((\alpha+Ax)-\alpha\right)$ gives, with $\|A^{-1}\|_1$ the operator norm induced by $\|\cdot\|_1$,
	\begin{equation}\label{eq:drift-lower}
		\sum_i|f_i(x)|=\|\alpha+Ax\|_1\ge\frac{\|x\|_1}{\|A^{-1}\|_1}-\|\alpha\|_1 ,\qquad x\in\R^n,
	\end{equation}
	and hence $\int\|x\|_1\,\mu(\ud x)\le\|A^{-1}\|_1\left(H_2+\|\alpha\|_1\right)$.
\end{proof}

\section{Invasion rates and invariant means}\label{sec:rates}

Recall the vector of \textit{noise-corrected growth rates}
\begin{equation}\label{eq:def_beta}
	\beta:=\alpha-\tfrac12\diag(\Sigma)\in\R^n ,
\end{equation}
where $\Sigma=\Gamma\Gamma^\top$ is the covariance matrix in \cref{ass:VL-a}. Since $A$ is invertible (\cref{rem:principal}), the vector
\begin{equation}\label{eq:def_m}
	m:=-A^{-1}\beta\in\R^n
\end{equation}
is well defined. It is the interior equilibrium, which need not be nonnegative, of the deterministic Lotka--Volterra system \cref{eq:JansenODE}.

\begin{definition}[Mean and invasion rates]\label{def:lambda}
	For a probability measure $\mu$ on $\R^n_+$ with finite first moment, let $\mu(X):=\int x\,\mu(\ud x)\in\R^n_+$ be its mean, and define the \emph{invasion rate} of species $i\in[n]$ with respect to $\mu$ by $\lambda_i(\mu):=\int\left(\beta_i+(Ax)_i\right)\mu(\ud x)$. By \cref{eq:def_m}, the vector $\lambda(\mu)=(\lambda_1(\mu),\dots,\lambda_n(\mu))$ is
	\begin{equation}\label{eq:def_lambda}
		\lambda(\mu)=\beta+A\mu(X)=A\left(\mu(X)-m\right).
	\end{equation}
\end{definition}

By \cref{lem:HNass-mom}, $\lambda(\mu)$ is defined for every invariant probability measure $\mu$ of $X$, and it coincides with the Lyapunov exponents of \citet{hening2018coexistence}, which in the notation of \cref{eq:HNform} read $\int\left(f_i-\frac12\sigma_{ii}g_i^2\right)\ud\mu$.

The next lemma is the Lotka--Volterra version of \citet[Lemma~2.1]{hening2018coexistence}, which is stated there for ergodic measures: a species that is present under a stationary regime has zero invasion rate. We give a short self-contained proof that applies directly to any invariant measure.

\begin{lemma}\label{lem:loglimit}
	Let $J\subseteq[n]$ and let $\mu$ be an invariant probability measure of $X$ supported on $\R^{J,\circ}_+$. Then $\lambda_i(\mu)=0$ for every $i\in J$.
\end{lemma}

\begin{proof}
	Let $X(0)\sim\mu$ so that $X$ is a stationary process under $\P_\mu$. By \cref{lem:HNass-wp}, $\R^{J,\circ}_+$ is invariant  and $\mu(\R^{J,\circ}_+)=1$. Then, for all $t\ge0$ and $\P_\mu$ almost surely, $X_i(t)>0$ if $i \in J$ while $X_i(t)=0$ if $i \notin J$. Now fix $i \in J$ and apply It\^o's formula to $\log X_i$,
	\begin{equation}\label{eq:dlogX}
		\ud\log X_i(t)=\left(\beta_i+(AX(t))_i\right)\ud t
		+\sum_{k=1}^n\Gamma_{ik}\,\ud B_k(t).
	\end{equation}
	Integrating,
	\begin{equation}\label{eq:loglimit-proof}
		\frac{\log X_i(t)-\log X_i(0)}{t}
		=\frac1t\int_0^t\left(\beta_i+(AX(s))_i\right)\ud s
		+\frac1t\sum_{k=1}^n\Gamma_{ik}B_k(t).
	\end{equation}
	The last term tends to $0$ almost surely by the strong law of large numbers for Brownian motion. The function $g(x):=\beta_i+(Ax)_i$ is $\mu$-integrable by \cref{lem:HNass-mom}. By Birkhoff's ergodic theorem for the stationary process $X$ under $\P_\mu$ \citep[Ch.~10]{kallenberg2002foundations}, the time average $\frac1t\int_0^tg(X(s))\ud s$ converges $\P_\mu$-a.s.\ and in $L^1(\P_\mu)$ to a random variable $L$ with $\EXP_\mu L=\int g\,\ud\mu =\lambda_i(\mu)$. Hence the left-hand side of \cref{eq:loglimit-proof} converges to $L$ almost surely.

	On the other hand, by stationarity, $\log X_i(t)$ has the same law as $\log X_i(0)$ for every $t$, and this is a finite random variable because $x_i>0$ for $\mu$-a.e.\ $x$. Hence, for every $\varepsilon>0$, as $t \to \infty$,
	\[
		\P_\mu\left(|\log X_i(t)|>\varepsilon t\right)
		=\mu\left\{x:\ |\log x_i|>\varepsilon t\right\}\longrightarrow0,
	\]
	so the left-hand side of \cref{eq:loglimit-proof} converges to $0$ in probability. Since it also converges to $L$ almost surely, and limits in probability are unique, $L=0$ almost surely. Therefore $\lambda_i(\mu)=\EXP_\mu L=0$.
\end{proof}

\Cref{lem:loglimit} turns the invasion rates of the present species into linear equations for the mean. To solve them, we introduce the equilibria of the subcommunities.

\begin{definition}[Face means]\label{def:mhat}
	For a nonempty $J\subseteq[n]$, let
	\begin{equation}\label{eq:def_mhat}
		\widehat m^{J}:=-(A^{J})^{-1}\beta^{J}\in\R^{|J|},
		\qquad
		m^{J}:=\ii_{J}\,\widehat m^{J}\in\R^n ,
	\end{equation}
	which are well defined because $A^J$ is invertible (\cref{rem:principal}). In particular $m^{[n]}=m$ and we set $m^\emptyset:=0$.
\end{definition}

For $n=3$ and ergodic measures, the following is \citet[Prop.~8.2]{hening2022classification}.

\begin{proposition}\label{prop:mean}
	Let $J\subseteq[n]$ and let $\mu$ be an invariant probability measure of $X$ supported on $\R^{J,\circ}_+$. Then $\mu(X)=m^{J}$ and $\widehat m^{J}>0$.
\end{proposition}

\begin{proof}
	If $J=\emptyset$, then $\mu=\delta_0$ and all assertions are immediate. Let $J\neq\emptyset$. By \cref{lem:HNass-mom}, $\mu(X)$ is finite, and since $\mu$ is supported on $\R^{J,\circ}_+$, $\mu(X)_i=0$ for $i\notin J$, that is, $\mu(X)=\ii_J\,\mu(X)^J$. By \cref{lem:loglimit}, $\beta_i+(A\mu(X))_i=0$ for every $i\in J$, which reads
	\[
		\beta^{J}+A^{J}\mu(X)^{J}=0 .
	\]
	Since $A^J$ is invertible (\cref{rem:principal}), $\mu(X)^{J}=-(A^{J})^{-1}\beta^{J}=\widehat m^{J}$, and hence $\mu(X)=\ii_J\widehat m^J=m^{J}$. Each component of $\widehat m^{J}=\mu(X)^J$ is positive, because $x_i>0$ for $\mu$-a.e.\ $x$ and every $i\in J$.
\end{proof}

For $J=[n]$, \cref{prop:mean} already gives one direction of our main result, without any use of the persistence theory.

\begin{corollary}\label{cor:necessity}
	If $X$ has an invariant probability measure $\pi$ with $\pi(\R^{n,\circ}_+)=1$, then its mean $\pi(X)=m$ is given by \cref{eq:def_m}, and $m>0$.
\end{corollary}

\section{Characterization of invariant coexistence}\label{sec:main}

The sufficiency part of our coexistence criterion rests on the persistence theorem of \citet{hening2018coexistence}, which is formulated in terms of the invariant measures that describe the possible stationary regimes with at least one species absent. Let
\begin{equation}\label{eq:def_M}
	\mathcal M:=\left\{\mu:\ \mu\ \text{is an ergodic invariant probability
		measure of } X \text{ with } \mu(\partial\R^n_+)=1\right\}.
\end{equation}
Note that, since the origin is absorbing for \cref{eq:SDE}, the point mass $\delta_0$ is invariant and, being a point mass, ergodic; hence $\delta_0\in\mathcal M$ always. Also, by \cref{lem:HNass-mom}, every $\mu\in\mathcal M$ has a finite first moment, so that the invasion rate $\lambda(\mu)$ is defined by \cref{eq:def_lambda}. 

We now state the persistence theorem of \citet{hening2018coexistence} in the weighted form that we use. It applies to our setting because \cref{eq:SDE} satisfies their Assumption~1.1 by \cref{lem:HNass-wp}.

\begin{theorem}[{\citealp[Theorem~1.1 and Sect.~4]{hening2018coexistence}}]\label{thm:HN}
	If there is $p\in\R^n$ with $p>0$ such that
	\begin{equation}\label{cond:HN}
		\inf_{\mu\in\mathcal M}\,p^\top\lambda(\mu)>0,
	\end{equation}
	then $X$ is \emph{strongly stochastically persistent}: there is a unique invariant probability measure $\pi$ with $\pi(\R^{n,\circ}_+)=1$, and the transition probabilities of $X$ converge to $\pi$ in total variation, exponentially fast, from every initial condition $x\in\R^{n,\circ}_+$:
	\begin{equation}\label{eq:HN-TV}
		\lim_{t\to\infty}\left\|\P_x\left(X(t)\in\cdot\,\right)-\pi\right\|_{\mathrm{TV}}=0 .
	\end{equation}
\end{theorem}

Theorem~1.1 of \citet{hening2018coexistence} is stated under the condition that $\max_i\lambda_i(\nu)>0$ for every finite convex combination $\nu=\sum_k\theta_k\mu_k$ of elements of $\mathcal M$ (their Assumption~1.2). \Cref{cond:HN} implies it: since \cref{eq:def_lambda} is linear in the measure, $p^\top\lambda(\nu)=\sum_k\theta_k\,p^\top\lambda(\mu_k)>0$, so $\lambda_i(\nu)>0$ for some $i$. Conversely, the two conditions are equivalent by a minimax argument \citep[Lemma~4]{schreiber2011persistence}; see also \citet[Sect.~4]{hening2018coexistence} and \citet[Remark~2.7]{hening2021general}. In either form, \cref{thm:HN} gives a sufficient condition only. In what remains of this section we show that, in our setting, \cref{cond:HN} holds with $p=m$ whenever $m>0$. Together with \cref{cor:necessity}, this completes our characterization.

The following lemma is the key step. It shows that every boundary regime lives on a single face of the orthant, and that its invasion rates, weighted by $m$, form a quadratic form in the distance between its mean and $m$.

\begin{lemma}\label{lem:weighted}
	For every $\mu\in\mathcal M$ there is a unique $J(\mu)\subsetneq[n]$ such that $\mu$ is supported on $\R^{J(\mu),\circ}_+$. Moreover, $\widehat m^{J(\mu)}>0$ and
	\begin{equation}\label{eq:weighted-mu}
		m^\top\lambda(\mu)=-\left(m^{J(\mu)}-m\right)^\top A\left(m^{J(\mu)}-m\right) .
	\end{equation}
\end{lemma}

\begin{proof}
	By \cref{lem:HNass-wp}, the sets $\R^{J,\circ}_+$, $J\subsetneq[n]$, form a partition of $\partial\R^n_+$ into invariant sets. Since $\mu$ is ergodic, $\mu(\R^{J,\circ}_+)\in\{0,1\}$ for every $J\subsetneq[n]$, and since $\mu(\partial\R^n_+)=1$, exactly one of these sets has $\mu$-measure one; this defines $J(\mu)$. For instance, $J(\delta_0)=\emptyset$. \Cref{prop:mean}, applied with $J=J(\mu)$, gives $\mu(X)=m^{J(\mu)}$ and $\widehat m^{J(\mu)}>0$, so that, by \cref{eq:def_lambda},
	\begin{equation}\label{eq:prop-lambda-mu}
		\lambda(\mu)=A\left(m^{J(\mu)}-m\right)=\beta+Am^{J(\mu)} .
	\end{equation}

	For the last identity, let $J\subseteq[n]$ and write $\Delta^J:=m^J-m$. By \cref{eq:def_m}, $\beta=-Am$, so $m^\top(\beta+Am^J)=m^\top A\Delta^J$. By \cref{eq:def_mhat}, $(Am^J)_i=(A^J\widehat m^J)_i=-\beta_i$ for $i\in J$, so $(A\Delta^J)_i=(Am^J)_i-(Am)_i=-\beta_i+\beta_i=0$ for $i\in J$, while $m^J_i=0$ for $i\notin J$. Hence $(m^J)^\top A\Delta^J=0$, and
	\[
		m^\top\left(\beta+Am^J\right)=\left(m-m^J\right)^\top A\Delta^J
		=-\left(\Delta^J\right)^\top A\,\Delta^J .
	\]
	Taking $J=J(\mu)$ and using \cref{eq:prop-lambda-mu} gives \cref{eq:weighted-mu}.
\end{proof}

We can now state and prove our coexistence criterion.

\begin{theorem}\label{thm:main}
	Assume \cref{ass:VL}, with any $D$ as in \cref{eq:VL}, and let $m=-A^{-1}\beta$. Then $X$ has an invariant probability measure $\pi$ with $\pi(\R^{n,\circ}_+)=1$ if and only if $m>0$. In that case, $X$ is strongly stochastically persistent: $\pi$ is the unique invariant probability measure with $\pi(\R^{n,\circ}_+)=1$, it is ergodic, its mean is $\pi(X)=m$, and the transition probabilities of $X$ converge to $\pi$ in total variation, exponentially fast, from every $x\in\R^{n,\circ}_+$.
\end{theorem}

\begin{proof}
	Let $\tilde X=DX$ and $\tilde A=AD^{-1}$ be as in \cref{eq:VL-reduced}. Since $d_i>0$, the map $x\mapsto Dx$ is a bijection of $\R^n_+$ onto itself that preserves $\R^{n,\circ}_+$ and maps invariant sets of $X$ to invariant sets of $\tilde X$. Hence $\pi\mapsto\pi\circ D^{-1}$ is a bijection from the invariant probability measures of $X$ onto those of $\tilde X$ that preserves ergodicity and full measure on $\R^{n,\circ}_+$; it also preserves total variation distances between laws, so the convergence statement transfers, with the same rate. Since $\alpha$ and $\Gamma$ are unchanged, so is $\beta$, and $\tilde m=-\tilde A^{-1}\beta=-DA^{-1}\beta=Dm$. Thus $\tilde m>0$ if and only if $m>0$, and $\pi$ has mean $m$ if and only if $\pi\circ D^{-1}$ has mean $\tilde m$. Every assertion of the theorem for $X$ is therefore equivalent to the same assertion for $\tilde X$, and, by \cref{eq:VL-reduced}, we may assume \cref{eq:assA}.

	If $\pi$ exists, then $m=\pi(X)>0$ by \cref{cor:necessity} and necessity follows. For sufficiency, assume $m>0$. Let $\mu\in\mathcal M$ and let $J=J(\mu)\subsetneq[n]$ be the face of its support. Denote $\Delta^J:=m^J-m$. Choosing $i\notin J$ gives $\Delta^J_i=-m_i<0$, so $\Delta^J\neq0$, and by \cref{lem:weighted}, $m^\top\lambda(\mu)=-(\Delta^J)^\top A\Delta^J>0$, because $A$ is negative definite by \cref{eq:negdef}. Since $\Delta^J\neq0$ for every $J\subsetneq[n]$ and there are finitely many such $J$, the same negative definiteness gives
	\begin{equation}\label{eq:punch}
		\inf_{\mu\in\mathcal M}m^\top\lambda(\mu)
		\ \ge\ \min_{J\subsetneq[n]}\left(-\left(\Delta^J\right)^\top A\,\Delta^J\right)>0 ,
	\end{equation}
	that is, \cref{cond:HN} holds with $p=m$, and \cref{thm:HN} gives a unique invariant probability measure $\pi$ with $\pi(\R^{n,\circ}_+)=1$, which attracts every $x\in\R^{n,\circ}_+$ exponentially fast in total variation. Moreover, $\pi$ is ergodic. Indeed, if $B$ were an invariant Borel set with $0<\pi(B)<1$, then $\pi(\cdot\cap B)/\pi(B)$ would be an invariant probability measure with full measure on $\R^{n,\circ}_+$ and different from $\pi$, contradicting uniqueness. Finally, $\pi(X)=m$ by \cref{cor:necessity}.
\end{proof}

We now apply \cref{thm:main} to subcommunities. Let $\emptyset\neq J\subseteq[n]$. For $X(0)\in\R^{J,\circ}_+$, the vector $\pp_JX(t) \in \R^{|J|}_+$ of the coordinates in $J$ of the solution of \cref{eq:SDE} solves a system like \cref{eq:SDE} with parameters $\alpha^J$, $A^J$ and the rows of $\Gamma$ indexed by $J$, whose covariance matrix $\Sigma^J$ is positive definite. This diffusion matrix has $|J|$ rows and $n$ columns, but the results above depend on the diffusion matrix only through the covariance matrix, so they apply to the subsystem. Moreover, if $A$ satisfies \cref{ass:VL} with diagonal matrix $D$, then $D^JA^J+(A^J)^\top D^J$ is a principal submatrix of $DA+A^\top D$, so \cref{ass:VL} holds for $A^J$. The noise-corrected growth rates of this subsystem are $\beta^J$, so the vector that plays the role of $m$ is $-(A^J)^{-1}\beta^J=\widehat m^J$.

\begin{corollary}\label{cor:sub}
	Assume \cref{ass:VL}. For every nonempty $J\subseteq[n]$, $X$ has an invariant probability measure supported on $\R^{J,\circ}_+$ if and only if $\widehat m^J>0$. In that case this measure, which we denote by $\pi^J$, is unique and ergodic, its mean is $m^J$, and the transition probabilities of $X$ converge to $\pi^J$ in total variation, exponentially fast, from every $x\in\R^{J,\circ}_+$; in particular $\pi^{[n]}$ is the measure $\pi$ of \cref{thm:main}. Moreover, the set $\mathcal{M}$ in \cref{eq:def_M} can be indexed by faces as
	\begin{equation}\label{eq:M-explicit}
		\mathcal M=\left\{\pi^J:\ J\subsetneq[n],\ \widehat m^J>0\right\},
	\end{equation}
	where $\pi^\emptyset:=\delta_0$ and, for $J=\emptyset$, the condition $\widehat m^\emptyset>0$ is void. Lastly, every invariant probability measure $\mu$ of $X$ is a convex combination $\mu=\sum_Jc_J\pi^J$, with $c_J=\mu(\R^{J,\circ}_+)$, over the sets $J\subseteq[n]$ with $\widehat m^J>0$.
\end{corollary}

\begin{proof}
	The first assertions follow from \cref{thm:main} applied to the subsystem, and $\pi^J$ has mean $\ii_J\widehat m^J=m^J$. In particular, $\pi^J\in\mathcal M$ whenever $J\subsetneq[n]$ and $\widehat m^J>0$. Conversely, by \cref{lem:weighted}, every $\mu\in\mathcal M$ is supported on $\R^{J(\mu),\circ}_+$, with $\widehat m^{J(\mu)}>0$, so $\mu=\pi^{J(\mu)}$ by uniqueness. This proves \cref{eq:M-explicit}. Finally, let $\mu$ be an invariant probability measure of $X$. By \cref{lem:HNass-wp}, the sets $\R^{J,\circ}_+$, $J\subseteq[n]$, partition $\R^n_+$ into invariant sets, so, whenever $c_J=\mu(\R^{J,\circ}_+)>0$, the normalized restriction $\mu(\cdot\cap\R^{J,\circ}_+)/c_J$ is an invariant probability measure supported on $\R^{J,\circ}_+$. By the first part, and since $\delta_0$ is the only probability measure supported on $\R^{\emptyset,\circ}_+=\{0\}$, we get $\widehat m^J>0$ and this measure is $\pi^J$. Hence $\mu=\sum_Jc_J\pi^J$.
\end{proof}

\section{Long-run behavior without invariant coexistence}\label{sec:saturated}

In this section we determine the long-run behavior of the solutions of \cref{eq:SDE} when the coexistence condition $m>0$ of \cref{thm:main} fails. We characterize the set $S$ of species that survive in the long run, and show that the occupation measure of $X$ converges almost surely to an invariant probability measure supported on $\R^{S,\circ}_+$. Generically, the species outside $S$ go extinct exponentially fast, so that the paths of $X$ approach the face $\R^S_+$. 

The solution to this problem is suggested by existing theory for the deterministic system \cref{eq:JansenODE}. Following \citet[Section~13.4]{hofbauer1998evolutionary}, a \emph{saturated nonnegative equilibrium} of \cref{eq:JansenODE} is a vector $\bar x\ge0$ such that $\Diag(\bar x)(\beta+A\bar x)=0$ and $(\beta+A\bar x)_i\le0$ whenever $\bar x_i=0$, that is, such that no absent species has a positive per-capita growth rate at $\bar x$. Equivalently, $\bar x$ solves the linear complementarity problem $\bar x\ge0$, $\beta+A\bar x\le0$, $\bar x^\top(\beta+A\bar x)=0$. Under \cref{ass:VL}, there is exactly one saturated nonnegative equilibrium $\bar x$ for every $\beta\in\R^n$, and it attracts every solution of \cref{eq:JansenODE} that starts in $\R^{n,\circ}_+$ \citep[Lemma~2, Theorem~1 and Corollary~2]{takeuchi1980existence}. We use $\bar x$ to define the face of $\R^n_+$ on which the long-run behavior of \cref{eq:SDE} takes place.

\begin{definition}\label{def:saturated}
	Let $\bar x$ be the unique saturated nonnegative equilibrium of \cref{eq:JansenODE}, and suppose that $\bar x\in\R^{S,\circ}_+$ for some $S\subsetneq[n]$. We call $\R^S_+$ the \emph{saturated face} for \cref{eq:SDE}.
\end{definition}

Thus $S$ is the set of species present in $\bar x$. If all species were present, namely $\bar x>0$, then $\beta+A\bar x=0$, so $\bar x=m$; conversely, if $m>0$, then $m$ is a saturated nonnegative equilibrium, so $\bar x=m$. Hence $\bar x>0$ if and only if $m>0$. Equivalently, the saturated face exists if and only if our invariant coexistence condition $m>0$ fails, in which case, $\bar x=m^S$. Indeed, since $\bar x_j=0$ for $j\notin S$, the rows $i\in S$ of $\Diag(\bar x)(\beta+A\bar x)=0$ read $\beta^S+A^S\pp_S\bar x=0$, so $\pp_S\bar x=\widehat m^S$ by \cref{def:mhat}. Hence, by the definition of $S$ and since $\bar x$ is saturated,
\begin{equation}\label{eq:saturated}
	\widehat m^S>0\qquad\text{and}\qquad\left(\beta+Am^S\right)_i\le0\quad\text{for all } i\notin S.
\end{equation}

By \cref{cor:sub}, since $\widehat m^S>0$, the face $\R^S_+$ carries a unique stationary regime $\pi^S$, and, by \cref{eq:prop-lambda-mu}, $\lambda(\pi^S)=\beta+Am^S$. Since $(\beta+Am^S)_i=0$ for $i\in S$, the condition in \cref{eq:saturated}, which says that no absent species can invade $\pi^S$, is equivalent to $\lambda(\pi^S)\le0$. The key step is the following counterpart of \cref{lem:weighted}, which shows that \cref{eq:saturated} characterizes $\pi^S$: an invariant measure that no species can invade is the stationary regime associated with the saturated nonnegative equilibrium of \cref{eq:JansenODE}.

\begin{lemma}\label{lem:noninvadable}
	Let $\mu$ be an invariant probability measure of $X$ such that $\lambda(\mu)\le0$. Then there is a unique $J(\mu)\subseteq[n]$ such that $\mu$ is supported on $\R^{J(\mu),\circ}_+$. Moreover, $\mu=\pi^{J(\mu)}$ and $m^{J(\mu)}=\bar x$. In particular, $\mu=\pi$ if $m>0$, and $\mu=\pi^S$ if $m>0$ fails.
\end{lemma}

\begin{proof}
	By \cref{cor:sub}, $\mu=\sum_Jc_J\pi^J$, where $c_J=\mu(\R^{J,\circ}_+)$, $\sum_Jc_J=1$, and the sum runs over the sets $J\subseteq[n]$ with $\widehat m^J>0$. By \cref{prop:mean}, $\mu(X)=\sum_Jc_Jm^J\ge0$. Consider the function $q(x):=x^\top(\beta+Ax)$. For every $J$, $q(m^J)=0$, because $(\beta+Am^J)_i=0$ for $i\in J$ and $m^J_i=0$ for $i\notin J$. Since $\sum_Jc_J=1$, expanding the quadratic forms gives an identity of the same type as \cref{eq:weighted-mu},
	\begin{equation}\label{eq:noninv-q}
		q\left(\mu(X)\right)=q\left(\mu(X)\right)-\sum_Jc_J\,q(m^J)=-\sum_Jc_J\left(m^J-\mu(X)\right)^\top A\left(m^J-\mu(X)\right).
	\end{equation}
	The right-hand side is nonnegative by \cref{eq:negdef}, while, by \cref{eq:def_lambda}, $q(\mu(X))=\mu(X)^\top\lambda(\mu)\le0$, because $\mu(X)\ge0$ and $\lambda(\mu)\le0$. Hence both sides of \cref{eq:noninv-q} vanish and, by \cref{eq:negdef} again, $m^J=\mu(X)$ for every $J$ with $c_J>0$. This means that $c_J=1$ for a single $J(\mu)$ so that $\mu$ is supported on $\R^{J(\mu),\circ}_+$, and $\mu=\pi^{J(\mu)}$. We thus have that $m^{J(\mu)} \geq 0$, $q(m^{J(\mu)})=0$, and $\beta+Am^{J(\mu)}=\lambda(\mu)\le0$ by \cref{eq:def_lambda}. Namely, $m^{J(\mu)}$ solves the linear complementarity problem for \cref{eq:JansenODE} and therefore $m^{J(\mu)} = \bar{x}$ is its saturated nonnegative equilibrium. If $m>0$, then $\bar x=m$, so $m^{J(\mu)}=m$, which forces $J(\mu)=[n]$ and $\mu=\pi$. If $m>0$ fails, then $\bar x\in\partial\R^n_+$, and $J(\mu)$ is the set $S$ of species present in $m^{J(\mu)}= \bar x$ and $\mu = \pi^S$.
\end{proof}

The next theorem shows that $\pi^S$ is the long-run regime of $X$. Its proof is based on the extinction theory developed in \citet[Section~5]{hening2018coexistence}, combined with \cref{lem:noninvadable}.

\begin{theorem}\label{thm:saturated}
	Assume \cref{ass:VL}, suppose that $m>0$ fails, and let $\R^S_+$ be the saturated face. For $t>0$, let $\Pi_t:=\frac1t\int_0^t\delta_{X(s)}\ud s$ denote the occupation measure of $X$. Then, for every $X(0)\in\R^{n,\circ}_+$, almost surely, $\Pi_t$ converges weakly to $\pi^S$ as $t\to\infty$, and
	\begin{equation}\label{eq:sat-limits}
		\lim_{t\to\infty}\frac1t\int_0^tX(s)\ud s=m^S,\qquad\lim_{t\to\infty}\frac{\log X_i(t)}{t}=\lambda_i(\pi^S)=\left(\beta+Am^S\right)_i\le0,\quad i\in[n].
	\end{equation}
	In particular, every species $i\notin S$ with $\lambda_i(\pi^S)<0$ goes extinct exponentially fast.
\end{theorem}

\begin{proof}
	The rescaling $\tilde X=DX$ of \cref{eq:VL-reduced} preserves $\beta$, the invasion rates, the saturated face and weak convergence of occupation measures, so, as in the proof of \cref{thm:main}, we may take $D=I$. We first check that Assumption~1.4 of \citet{hening2018coexistence} holds for \cref{eq:SDE}. Indeed, with $f$ and $g$ as in \cref{eq:HNform},
	\begin{equation}\label{eq:HN14}
		\lim_{\|x\|_1 \to \infty} \frac{\|x\|_1^{\delta_1}\sum_ig_i^2(x)}{1+\sum_i\left(|f_i(x)|+g_i^2(x)\right)}= \lim_{\|x\|_1 \to \infty} \frac{n\|x\|_1^{\delta_1}}{1+n+\|\alpha+Ax\|_1} = 0
	\end{equation}
	for any $\delta_1 \in (0,1)$, since by \cref{eq:drift-lower} the denominator grows at least linearly in $\|x\|_1$.
	
	Fix $X(0)=x\in\R^{n,\circ}_+$. By Lemmas~5.4 and~5.7 of \citet{hening2018coexistence}, $\P_x$-almost surely, the family $(\Pi_t)_{t\ge1}$ is tight and each of its weak limit points is an invariant probability measure of $X$. Moreover, $\P_x$-almost surely, for every sequence $t_k\to\infty$ and every probability measure $\mu$ such that $\Pi_{t_k}\to\mu$ weakly, we have $\frac1{t_k}\int_0^{t_k}X(s)\ud s\to\mu(X)$. This last convergence does not follow from weak convergence alone because $y\mapsto y_j$ is unbounded. It follows from Lemma~5.6 of \citet{hening2018coexistence} applied to $h(y) = y_j$, since $|y_j|\le\|A^{-1}\|_1\left(\sum_i|f_i(y)|+\|\alpha\|_1\right)$ by \cref{eq:drift-lower}. Lastly, eq.~(5.22) of \citet{hening2018coexistence} gives $\P_x(\limsup_{t\to\infty}\log X_i(t)/t\le0)=1$ for every $i\in[n]$. 
	
	For $x\in\R^{n,\circ}_+$, by the argument above, we can, $\P_x$-almost surely, pick a weak limit $\mu = \lim_{k} \Pi_{t_k}$ that is invariant and a path with $\limsup_{k}\log X_i(t_k)/{t}_k\le0$, and (by the law of large numbers) $\Gamma B(t_k)/{t_k}\to0$. For such a path $\tfrac{1}{t_k} \int_0^{t_k} (\beta + A X(s)) \ud s \to (\beta + A \mu(X))$; thus \cref{eq:loglimit-proof} and \cref{def:lambda} give
	\begin{equation}\label{eq:sat-loglim}
		\lim_{k\to\infty}\frac{\log X_i(t_k)}{t_k}=\beta_i+\left(A\mu(X)\right)_i=\lambda_i(\mu)\le0,\quad i\in[n].
	\end{equation}	
	
	Since $m>0$ fails, \cref{lem:noninvadable} gives $\mu=\pi^S$ and, because $\mu$ was an arbitrary weak limit point of the tight family $(\Pi_t)_{t\ge1}$, $\Pi_t\to\pi^S$ weakly. The limits in \cref{eq:sat-limits} follow by applying the properties above and \cref{eq:sat-loglim} along arbitrary sequences $t_k\to\infty$, since $\pi^S(X)=m^S$. In particular, if $i\notin S$ and $\lambda_i(\pi^S)<0$, then for every $\varepsilon>0$ we have $X_i(t)\le e^{(\lambda_i(\pi^S)+\varepsilon)t}$ for all sufficiently large $t$, so $X_i(t)\to0$ exponentially fast.
\end{proof}

\begin{remark}\label{rem:HN-extinction}
	In the borderline case $(\beta+Am^S)_i=0$ for some $i\notin S$, \cref{thm:saturated} gives $\log X_i(t)/t\to0$ but not $X_i(t)\to0$, which may fail: for $n=1$ and $\beta=0$, $\log X$ is null recurrent. When the inequalities $(\beta+Am^S)_i\le0$, $i\notin S$, in \cref{eq:saturated} are strict, the weak convergence of $\Pi_t$ to $\pi^S$ and the limits of $\log X_i(t)/t$ for $i\notin S$ also follow from \citet[Theorem~1.3]{hening2018coexistence}. In their notation, $\pi^S$ is the only element of $\mathcal M_1$: it satisfies their condition (1.6) by the strict inequalities and their condition (1.7) by the proof of \cref{thm:main} applied to the subsystem $S$, and no other $\pi^J\in\mathcal M$ satisfies (1.6), since $m^J$ would then be a saturated nonnegative equilibrium different from $\bar x$. Their Assumption~1.5 follows from \cref{lem:noninvadable}: if $\nu\in\Conv(\mathcal M\setminus\{\pi^S\})$ had $\lambda(\nu)\le0$, then $\nu=\pi^S$, which is impossible because $\nu(\R^{S,\circ}_+)=0$. In the borderline case, where $(\beta+Am^S)_i=0$ for some $i\notin S$, no element of $\mathcal M$ satisfies (1.6), and their theorem does not apply.
\end{remark}

\section{Discussion}\label{sec:discussion}

We have shown that, for \cref{eq:SDE} with a Volterra--Lyapunov stable interaction matrix and nondegenerate noise, invariant coexistence is equivalent to $m=-A^{-1}\beta>0$, the feasibility criterion for the deterministic system \cref{eq:JansenODE}, and that, when $m>0$ fails, the long-run behavior of $X$ is characterized in terms of the saturated nonnegative equilibrium of \cref{eq:JansenODE} (\cref{thm:saturated}; see also \cref{sec:disc-sub}). The coexistence criterion is explicit: it requires solving one linear system and inspecting the signs of its solution. Both directions exploit the same key feature of \cref{eq:SDE}: the per-capita drift $\alpha+Ax$ is affine in $x$ and the covariance $\Sigma$ of the per-capita noise does not depend on $x$, so that $\lambda(\mu)=A(\mu(X)-m)$ depends on $\mu$ only through its mean (\cref{eq:def_lambda}). The necessity argument (\cref{lem:loglimit,prop:mean}) uses only well-posedness, finite first moments and the invertibility of $A$, so $m>0$ is necessary for invariant coexistence whenever these hold, even without \cref{ass:VL-b}. In the rest of this section we discuss the scope of \cref{ass:VL} (\cref{sec:disc-scope}), the role of environmental noise, with an example in which noise on one species increases the number of surviving species (\cref{sec:disc-jensen} and \cref{ex:noise}), the relation of \cref{eq:punch} to Hofbauer's criterion and to the Lyapunov function of Volterra and Goh (\cref{sec:disc-weights}), and the expression of the long-run behavior of \cref{eq:SDE} in terms of the equilibria of \cref{eq:JansenODE} (\cref{sec:disc-sub}).

\subsection{Scope of \texorpdfstring{\cref{ass:VL}}{Assumption 1}}\label{sec:disc-scope}

\Cref{ass:VL-b} covers a wide family of ecological models, provided that the self-limitation ($A_{ii}<0$) of every species is strong relative to the interspecific effects. It does not prescribe the type of interactions, but it includes many predator--prey systems, which are the main motivation of this work.

In a predator--prey community, each pair of species $i\neq j$ either does not interact, $A_{ij}=A_{ji}=0$, or interacts trophically, $A_{ij}A_{ji}<0$; we then call $A$ a predator--prey matrix, and we call it self-limited if, in addition, $A_{ii}<0$ for all $i$, which models intraspecific competition. Such a matrix defines a predator--prey web: the directed graph on $[n]$ with an edge $i\to j$ whenever $A_{ij}>0>A_{ji}$, meaning that species $i$ preys on species $j$. \Cref{ass:VL-b} does not include all self-limited predator--prey matrices, and it also covers interactions other than predation. We now give specific conditions under which it holds.

\smallskip\noindent\emph{(i) Conservative predator--prey webs.} Following \citet{volterra1931lecons}, we say that $A$ is \emph{conservative} if there is $d\in\R^{n,\circ}_+$ such that
\begin{equation}\label{eq:conservative}
	d_iA_{ij}+d_jA_{ji}=0\qquad\text{for all } i\neq j ,
\end{equation}
that is, if the off-diagonal part of $\Diag(d)A$ is skew-symmetric. Since $d>0$, \cref{eq:conservative} forces $A$ to be a predator--prey matrix. If $A$ is conservative and self-limited, then \cref{ass:VL-b} holds with $D=\Diag(d)$, because
\[
	DA+A^\top D=2\Diag(d_1A_{11},\dots,d_nA_{nn})\prec0 .
\]
Conservativeness is a condition on the topology of the web and on the ratios of the interaction coefficients, not on their size. Indeed, \cref{eq:conservative} prescribes $d_j/d_i=-A_{ij}/A_{ji}$ along every edge, so a predator--prey matrix is conservative if and only if
\begin{equation}\label{eq:cycle}
	\prod_{k=1}^{\ell}\left(-\frac{A_{i_ki_{k+1}}}{A_{i_{k+1}i_k}}\right)=1
\end{equation}
around every cycle $i_1,i_2,\dots,i_\ell,i_{\ell+1}=i_1$ of the underlying undirected graph of the web \citep[Prop.~3.2]{fernandes1995hamiltonian}. In particular, every predator--prey matrix whose web is a forest, such as a food chain or a tree-shaped food web, is conservative \citep{solimano1982graph}. Hence every self-limited predator--prey matrix on a forest satisfies \cref{ass:VL-b}, however strong its interactions. Webs with cycles, such as webs with omnivory, also satisfy the assumption when every cycle is balanced in the sense of \cref{eq:cycle}.

\smallskip\noindent\emph{(ii) Predator--prey webs that are not conservative.} If some cycle is unbalanced, a self-limited predator--prey matrix is not conservative, but it may still satisfy \cref{ass:VL-b} when self-limitation is strong enough. For instance, for the omnivory web $1\to2$, $1\to3$, $2\to3$, the matrix
\begin{equation}\label{eq:omni-example}
	A=\begin{pmatrix}-10&1&1\\-1&-10&2\\-1&-1&-10\end{pmatrix}
	\quad\text{satisfies}\quad
	A+A^\top=\begin{pmatrix}-20&0&0\\0&-20&1\\0&1&-20\end{pmatrix}\prec0 ,
\end{equation}
so \cref{eq:assA} holds with $d=\mathbf 1$. However, $A$ is not conservative: the product in \cref{eq:cycle} around the cycle $1,2,3,1$ equals $1\cdot2\cdot1=2$. Without sufficient self-limitation \cref{ass:VL-b} may fail: with the same off-diagonal entries and $A_{ii}=-1/10$, the matrix $A$ has a pair of eigenvalues with positive real part, whereas \cref{eq:assA} forces $A$ to be Hurwitz (see below).

\smallskip\noindent\emph{(iii) Other types of interaction.} \Cref{ass:VL-b} does not require predator--prey relationships. For instance, $A+A^\top\prec0$ holds whenever self-limitation dominates the interactions in the sense that $2|A_{ii}|>\sum_{j\neq i}\left(|A_{ij}|+|A_{ji}|\right)$ for every $i$, by Gershgorin's theorem applied to $A+A^\top$. This allows weak competition ($A_{ij},A_{ji}<0$) and weak mutualism ($A_{ij},A_{ji}>0$), which are never conservative because of their signs. Two cases can be described exactly. For two species, \cref{ass:VL-b} holds if and only if $A_{11},A_{22}<0$ and $\det A>0$, as follows by choosing $D=\Diag(1,t)$ with a suitable $t>0$. In larger communities the same inequality is necessary for every pair of species: for $J=\{i,j\}$, $\sign(\det A^J)=(-1)^{|J|}$ (\cref{rem:principal}) gives $\det A^J=A_{ii}A_{jj}-A_{ij}A_{ji}>0$. This excludes interspecific competition or mutualism that is strong compared with intraspecific competition. For cooperative communities, in which $A_{ij}\ge0$ for all $i\neq j$, \cref{eq:assA} holds if and only if $A$ is Hurwitz, that is, if and only if $-A$ is a nonsingular M-matrix \citep{berman1994nonnegative}.

Under \cref{ass:VL}, $V(x)=x^\top Dx$ is a strict Lyapunov function for $\dot x=Ax$, so $A$ is Hurwitz, and so is every $A^J$. Much more is true: by \citet{takeuchi1980existence}, \cref{ass:VL} implies that the deterministic Lotka--Volterra system $\dot x=\Diag(x)(\alpha+Ax)$ has, for every $\alpha\in\R^n$, a unique saturated equilibrium, which is globally stable. Hence models whose deterministic dynamics exhibit alternative stable states, sustained oscillations or heteroclinic cycles lie outside our scope. Neither the P-matrix property nor the Hurwitz property implies \cref{ass:VL}: the matrix $\left(\begin{smallmatrix}1&-2\\2&-2\end{smallmatrix}\right)$ is Hurwitz but
has a positive diagonal entry, and the omnivory matrix \cref{eq:omni-example} with $A_{ii}=-1/10$ has $-A$ a P-matrix but is not Hurwitz.

\subsection{Environmental noise and the It\^o shift}\label{sec:disc-jensen}

We have shown that invariant coexistence is completely determined by the vector $m = -A^{-1}\beta$. Hence environmental noise affects invariant coexistence, and the stationary means, only through the It\^o shift $\alpha_i\mapsto\beta_i=\alpha_i-\frac12\Sigma_{ii}$. By \cref{thm:saturated}, the same holds when $m>0$ fails, for the set $S$ of surviving species, their long-run mean $m^S$ and the exponential rates of extinction $(\beta+Am^S)_i$. In particular, the off-diagonal entries of $\Sigma$, which describe correlations between the environmental responses of different species, enter neither the criterion nor the means; they affect only the shape of the stationary distribution.

This ``variance reduces long-run growth'' effect can be seen as a consequence of Jensen's inequality \citep{jensen1906fonctions}, whose implications for populations in fluctuating environments are reviewed by \citet{ruel1999jensen} and \citet{denny2017fallacy}. It takes the form of the shift $\alpha\mapsto\beta$ and acts on each population separately. Under multiplicative fluctuations, long-run growth is governed by the average of the logarithmic growth rate, which is smaller than the logarithm of the average growth rate; for white noise this reduction is exactly the It\^o correction $\frac12\Sigma_{ii}$ in \cref{eq:dlogX}. These facts have been reported in the literature since the early works on randomly fluctuating population dynamics \citep[see][]{lewontin1969population,turelli1977random}. What is significant about the stochastic Lotka--Volterra model is that this is the only Jensen effect on invariant coexistence and on the stationary means. In the language of the coexistence theory of \citet{chesson1994multispecies,chesson2000mechanisms}, the community-level, fluctuation-dependent mechanisms that may promote or prevent coexistence are absent precisely because of the key feature of \cref{eq:SDE} described above: an affine per-capita drift rules out relative nonlinearity, and a noise term that does not depend on the densities makes the responses to the environment and to competition additive, which rules out the storage effect. Indeed, by \cref{eq:def_lambda}, $\lambda(\mu)=\EXP_{\mu} (\beta + AX) =\beta+A\,\EXP_{\mu} X$ for every $\mu$, so Jensen's inequality holds with equality, and the invasion rates, and hence the criterion for invariant coexistence, coincide with the per-capita growth rates of \cref{eq:JansenODE} evaluated at the means.

The effect of noise acting on a single species need not be harmful to the community. Increasing $\Sigma_{jj}$ by $\varepsilon$ changes $m$ by $\frac\varepsilon2$ times the $j$th column of $A^{-1}$, and the entries of $A^{-1}$ can have either sign, so noise acting on one species can raise the long-run means of others. By \cref{thm:main,thm:saturated}, it can also change which species survive, as the following example shows.

\begin{example}\label{ex:noise}
Consider the omnivory web $1\to2$, $1\to3$, $2\to3$ with
\begin{equation}\label{eq:noise-A}
	A=\begin{pmatrix}-1&2&1\\-2&-1&1\\-1&-1&-1\end{pmatrix},\qquad
	\beta=\begin{pmatrix}\beta_1\\-\frac15\\1\end{pmatrix},
\end{equation}
where $A$ is self-limited and conservative with $d=\mathbf 1$, so that $A+A^\top=-2I$. The vector $\alpha$ and the matrix $\Sigma\succ0$ are such that they give these values of $\beta$, and we regard $\beta_1=\alpha_1-\frac12\Sigma_{11}$ as a parameter that decreases as the environmental noise on the top predator increases. Then
\begin{equation}\label{eq:noise-m}
	m=\left(\tfrac27\beta_1+\tfrac25,\ -\tfrac37\beta_1-\tfrac15,\ \tfrac17\beta_1+\tfrac45\right)^\top,
\end{equation}
so invariant coexistence holds if and only if $-\frac75<\beta_1<-\frac7{15}$. For $-\frac7{15}<\beta_1<1$, the top predator is too abundant for the intermediate consumer to survive. The saturated face is $\R^{\{1,3\}}_+$, with $\bar x=\tfrac12\left(1+\beta_1,0,1-\beta_1\right)^\top$, and the invasion rate of the intermediate consumer is $\lambda_2(\pi^{\{1,3\}})=(\beta+A\bar x)_2=-\left(\frac32\beta_1+\frac7{10}\right)<0$, so, by \cref{thm:saturated}, $X_2(t)\to0$ exponentially fast almost surely. For $\beta_1<-\frac75$, the noise is too large for the top predator to survive. The saturated face is $\R^{\{2,3\}}_+$, with $\bar x=\left(0,\frac25,\frac35\right)^\top$, and $\lambda_1(\pi^{\{2,3\}})=(\beta+A\bar x)_1=\beta_1+\frac75<0$, so $X_1(t)\to0$ exponentially fast almost surely. Hence moderate noise on the top predator alone turns two surviving species into three coexisting ones. It does so by lowering the predator's long-run density, which releases the intermediate consumer. These regimes are shown in \cref{fig:noise}.
\end{example}

\begin{figure}[htbp]
	\centering
	\includegraphics[scale=0.7]{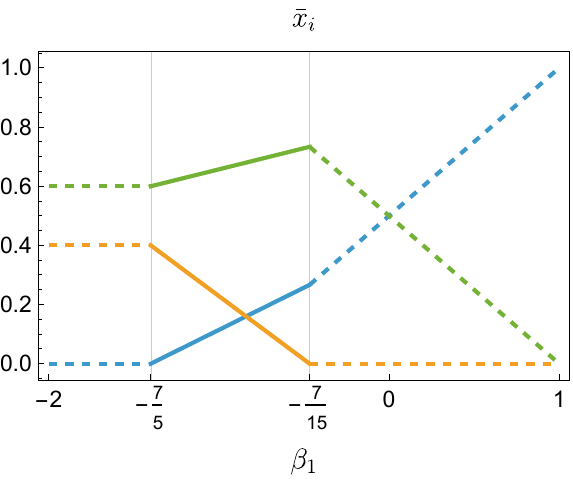} \quad
	\raisebox{-1.6pt}{\includegraphics[scale=0.7254]{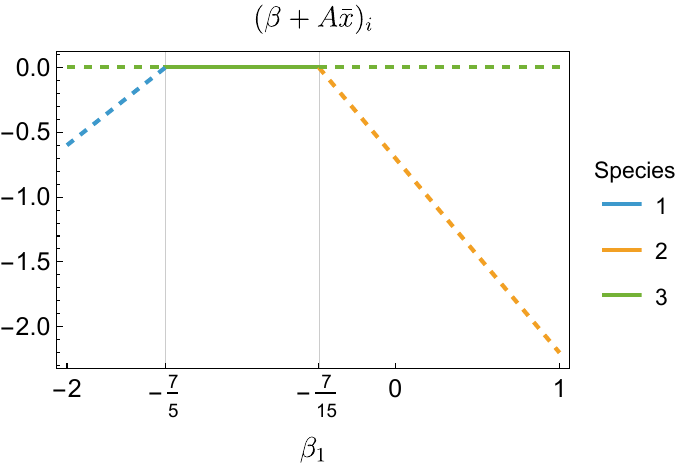}}
	\caption{Long-run behavior of the omnivory web of \cref{ex:noise} as a function of the noise-corrected growth rate $\beta_1=\alpha_1-\frac12\Sigma_{11}$ of the top predator, which decreases as the environmental noise on the top predator increases. Left: the long-run means $\bar x_i$, which are the limits of the time averages of $X_i$ in \cref{eq:unified}. Right: the long-run growth rates $(\beta+A\bar x)_i$, which are the almost sure limits of $\log X_i(t)/t$; for a species absent from $\bar x$ this is its exponential rate of extinction. Solid lines show the interval of invariant coexistence, where $\bar x=m$ and all growth rates vanish. Dashed lines show the regimes in which $m>0$ fails. In the right panel, the zero growth rates of different species overlap.}
	\label{fig:noise}
\end{figure}

\subsection{Relations to Hofbauer's criterion and the Lyapunov function of Volterra and Goh}\label{sec:disc-weights}

The inequality at the heart of the proof of \cref{thm:main} is \cref{eq:punch}, which says that, if $m>0$, then $m^\top\lambda(\mu)$ is bounded below by a positive constant over $\mathcal M$; that is, when weighted by $m$, the average of the invasion rates is positive, uniformly, in every stationary regime on the boundary.

We now argue that \cref{eq:punch} is the stochastic form of the average Lyapunov function criterion of \citet{hofbauer1981general}. For $q>0$, the function $Q(x)=\prod_ix_i^{q_i}$ vanishes exactly on $\partial\R^n_+$, and by \cref{eq:dlogX} the solution of \cref{eq:SDE} satisfies \begin{equation}\label{eq:sdelogQ}
	\ud\log Q(X)=q^\top(\beta+AX)\ud t+q^\top\Gamma\ud B .
\end{equation}
Take $q=m$. The average of the drift $m^\top(\beta+AX)$ with respect to a boundary regime $\mu\in\mathcal M$ is $m^\top\lambda(\mu)=\sum_{i\notin J(\mu)}m_i\lambda_i(\mu)$, because the species present in $\mu$ have zero invasion rates (\cref{lem:loglimit}); it is a weighted average of the invasion rates of the missing species. Since this average is positive by \cref{eq:punch}, while the community stays close to the regime $\mu$ (with the species outside $J(\mu)$ rare), the ergodic theorem makes the drift average out to a positive number over long times, and $\log Q(X(t))$ increases. The log-densities of the resident species do not drift, so this increase means that some missing species invades, and the process is pushed away from $\partial\R^n_+$. Since $m^\top\lambda$ is linear in the measure and the weights are fixed, the same holds when the community moves between several boundary regimes, and the uniform lower bound in \cref{eq:punch} makes this robust. The proof of \cref{thm:HN} in \citet{hening2018coexistence} makes this heuristic rigorous.

In the deterministic setting, the corresponding statement is Hofbauer's criterion \citetext{\citealp[Thm.~1]{hofbauer1981general}; \citealp[Thm.~12.2.1]{hofbauer1998evolutionary}}, which states that a dissipative system $\dot x_i=x_if_i(x)$ is permanent if, for some $q>0$, the time average of $\sum_iq_if_i$ along every orbit in $\partial\R^n_+$ is eventually positive. We apply it to the deterministic system \cref{eq:JansenODE}, which is dissipative because $A$ is negative definite by \cref{eq:negdef}. For dissipative Lotka--Volterra systems, Hofbauer's criterion needs to be checked only at the boundary equilibria: the system is permanent if there is $q>0$ such that $q^\top(\beta+A x^*)>0$ for every equilibrium $x^*\in\partial\R^n_+$. See \citet[][Theorem 3]{jansen1987permanence} and \citet[Exercise~13.6.3]{hofbauer1998evolutionary}. The boundary equilibria of \cref{eq:JansenODE} are the vectors $x^* = m^J$ with $J\subsetneq[n]$ and $\widehat m^J>0$. By \cref{eq:prop-lambda-mu}, the invasion rates $\lambda(\mu)=\beta+Am^{J(\mu)}$ of a boundary regime are the per-capita growth rates of \cref{eq:JansenODE} at the equilibrium $m^{J(\mu)}$, and by \cref{cor:sub} every such equilibrium carries exactly one boundary regime. Hence \cref{eq:punch}, restricted to $\mathcal M$, is exactly Jansen's condition for \cref{eq:JansenODE} with $q=m$. In our setting the weights in \cref{cond:HN} are explicit.

The weights $m$ in \cref{eq:punch} can also be identified with those of the classical Lyapunov function of \citet{volterra1931lecons} and \citet{goh1977global}. If $m>0$, the function
\begin{equation}
	V(x)=\sum_{i=1}^n\left(x_i-m_i-m_i\log\frac{x_i}{m_i}\right)
\end{equation}
satisfies $\dot V=(x-m)^\top A\,(x-m)<0$ for $x\neq m$, by \cref{eq:negdef}, along solutions of \cref{eq:JansenODE}. Up to an additive constant, $V(x)=\mathbf 1^\top x-\log Q(x)$ with $Q(x)=\prod_ix_i^{m_i}$, and, as shown above, $m^\top\lambda(\mu)$ is the long-run growth rate of $\log Q(X)$ near the support of $\mu$. The positivity in \cref{eq:punch} is that of $-\dot V$ evaluated at the mean of a boundary regime. Indeed, since $\mu(X)=m^{J(\mu)}$, \cref{eq:weighted-mu} reads $m^\top\lambda(\mu)=-(\mu(X)-m)^\top A(\mu(X)-m)$, which is positive by \cref{eq:negdef} because $\mu(X)\neq m$. The linear part $\mathbf 1^\top x$ is the function that controls large densities in \cref{lem:HNass}, where $c=\mathbf 1$. Using $V$ directly is less effective, because the generator $\mathcal L$ of \cref{eq:SDE} gives
\[
	\mathcal LV(x)=(x-m)^\top A\,(x-m)+\frac12\sum_{i=1}^n\Sigma_{ii}x_i ,
\]
which, by \cref{eq:negdef}, is negative outside a bounded set, but in general not near $\partial\R^n_+$ unless the noise is small. A Lyapunov argument based on $V$ therefore yields only sufficient conditions involving the noise intensities, whereas averaging $\log Q$ with respect to the boundary measures removes this restriction.

For a general $D$ as in \cref{eq:VL}, the same statements hold in the original coordinates with the weights $Dm$ in place of $m$, the function $V(x)=\sum_id_i\left(x_i-m_i-m_i\log(x_i/m_i)\right)$, and the quadratic forms $(x-m)^\top DA\,(x-m)$; these are the expressions above written for $\tilde X=DX$, whose mean is $\tilde m=Dm$.

\subsection{Long-run behavior and the equilibria of the shifted system}\label{sec:disc-sub}

In the classifications available for low dimensions and special topologies \citep[see][]{hening2018stochastic,hening2018persistence,hening2022classification}, deciding invariant coexistence through \cref{thm:HN} requires first determining $\mathcal M$, face by face starting from $\delta_0$, and then verifying \cref{cond:HN} case by case; the number of cases grows like $2^n$. \Cref{thm:main} replaces this procedure by the single condition $m>0$. Invariant coexistence of a community does not require its subcommunities to have stationary regimes. In \cref{ex:noise} with $-\frac75<\beta_1<-1$, the top predator cannot persist with the basal species alone, since $\widehat m^{\{1,3\}}_1=\frac{1+\beta_1}2<0$, yet $m>0$ and all three species coexist. In its proof, \cref{cond:HN} is verified without determining $\mathcal M$: by \cref{lem:weighted}, $m^\top\lambda(\mu)$ depends on $\mu$ only through its face $J(\mu)$, and \cref{eq:punch} bounds it over all proper faces, whether or not they carry a stationary regime. The description \cref{eq:M-explicit} of $\mathcal M$ is then a by-product of \cref{thm:main}, available when the boundary regimes themselves are of interest.

Read in this way, \cref{thm:main,cor:sub} express the community structure of \cref{eq:SDE} in terms of the equilibria of \cref{eq:JansenODE}. A subcommunity $J$ has a stationary regime if and only if $\widehat m^J=-(A^J)^{-1}\beta^J>0$, that is, if and only if \cref{eq:JansenODE} has a feasible equilibrium on the face $\R^J_+$; this is decided face by face, independently of the other faces, with no recursion. A species $i\notin J$ can invade that regime if and only if $(\beta+Am^J)_i>0$, the sign of its per-capita growth rate at the same equilibrium, and every invariant probability measure of $X$ is a mixture of the regimes $\pi^J$ obtained in this way. \Cref{thm:saturated} completes this picture. By \citet[Theorem~1 and Corollary~2]{takeuchi1980existence}, every solution of \cref{eq:JansenODE} that starts in $\R^{n,\circ}_+$ converges to the saturated nonnegative equilibrium $\bar x$. The proof of \cref{thm:saturated} uses the assumption that $m>0$ fails only when \cref{lem:noninvadable} is applied; if $m>0$, the same argument, with \cref{lem:noninvadable} giving $\mu=\pi$, shows that the occupation measures of $X$ converge weakly to $\pi$ almost surely. Hence, whether or not $m>0$, the long-run behavior of every solution of \cref{eq:SDE} that starts in $\R^{n,\circ}_+$ can be expressed through $\bar x$: almost surely,
\begin{equation}\label{eq:unified}
	\lim_{t\to\infty}\frac1t\int_0^tX(s)\ud s=\bar x,\qquad\lim_{t\to\infty}\frac{\log X_i(t)}{t}=\left(\beta+A\bar x\right)_i,\quad i\in[n],
\end{equation}
the occupation measures of $X$ converge to the stationary regime of the species present in $\bar x$, and the negative entries of $\beta+A\bar x$ are the exponential rates at which the other species go extinct.

\section*{Acknowledgments}

The author is grateful to Carlos Osorio, Alexander Valencia and Jose Quintero, whose interest in this subject sparked the project and who contributed through discussions during its development.

\paragraph{Use of AI tools.} This paper was written through a human--AI collaboration. Claude (Anthropic), with the Claude Opus 5.5 model used in Cowork mode, was used for literature search and analysis, numerical experimentation, the proposal of technical and conceptual strategies for proofs and mathematical arguments, and iterative revision of the readability, structure, grammar and English spelling of the manuscript. The author was responsible for the ideas, the structure of the presentation, the main directions of inquiry, the argumentative emphasis, the final selection and evaluation of the relevant literature, the conceptual judgments, and the verification of all mathematical derivations. The author takes full responsibility for the content of the paper and the correctness of its results.

\section*{Statements and Declarations}

\paragraph{Funding.} This work was funded by the Facultad de Ciencias of the Universidad Nacional de Colombia, Sede Medell\'{\i}n, through the project ``Resilience in interconnected systems'' (HERMES code 202010049694).

\paragraph{Competing interests.} The author has no relevant financial or non-financial interests to disclose.

\paragraph{Data availability.} Data sharing is not applicable to this article, as no datasets were generated or analysed during the current study.

\paragraph{Author contributions.} The author is the sole contributor to this work.

\bibliographystyle{plainnat}
\bibliography{pred_prey_arxiv}

\end{document}